\documentclass[reqno]{amsart}

\usepackage[utf8]{inputenc} 
\usepackage{amssymb}
\usepackage{graphicx}
\usepackage{latexsym}
\usepackage{stmaryrd}
\usepackage{enumitem}
\usepackage[bookmarks,hyperfootnotes=false]{hyperref}
\usepackage{multirow}
\usepackage{xcolor}
\usepackage{caption}
\colorlet{darkishRed}{red!80!black}
\colorlet{darkishBlue}{blue!60!black}
\colorlet{darkishGreen}{green!60!black}
\usepackage{hyperref}
\hypersetup{
pdftitle = {Duality theorems for stars and combs},
pdfsubject = {Duality theorems for stars and combs},
pdfauthor = {Carl Bürger and Jan Kurkofka},
pdfkeywords = {},
draft = false,
bookmarksopen=true}

\usepackage{theoremref} 
\usepackage{tikz}
\usepackage{tikz-cd}
\usetikzlibrary{calc,through,intersections,arrows, trees, positioning, decorations.pathmorphing, cd}
\usepackage{relsize}
\usepackage{comment}
\usepackage{xifthen}
\usepackage{mathrsfs}
\usepackage{svg}
\usepackage{mathtools}
\usepackage{nccmath}
\usepackage{pifont}

\newcommand{\cmark}{\ding{51}}%
\newcommand{\xmark}{\ding{55}}%

\newcommand{\SCnumberCite}[1]
{%
\ifthenelse{\equal{#1}{1}}{\cite{StarComb1StarsAndCombs}}{}%
\ifthenelse{\equal{#1}{2}}{\cite{StarComb2TheDominatedComb}}{}%
\ifthenelse{\equal{#1}{3}}{\cite{StarComb3TheUndominatedComb}}{}%
\ifthenelse{\equal{#1}{4}}{\cite{StarComb4TheUndominatingStar}}{}%
}

\newcommand{\SCnumberHand}[2]
{%
\ifthenelse{\equal{#1}{#2}}{\ding{43}\,}{}%
}

\newcommand{\SCintroList}[2]
{%
    \ifthenelse{\equal{#1}{#2}}{(this paper)}{\SCnumberCite{#1}}%
}

\newcommand{\SCintroDetermined}[1]
{%
    \ifthenelse{\equal{#1}{1}}{In this paper, we determine}{In the first paper of this series, we determined}%
}

\newcounter{quotecount}

\makeatletter
\newcommand*{\addFileDependency}[1]{
  \typeout{(#1)}
  \@addtofilelist{#1}
  \IfFileExists{#1}{}{\typeout{No file #1.}}
}
\makeatother

\newcommand{\minG}{\vert G\vert_{\Gamma}}

\newcommand{\Abs}[1]{\partial_{\Omega} {#1}}
\newcommand{\AbsC}[1]{\partial_{\Gamma} {#1}}
\newcommand{\crit}{\normalfont\text{crit}}

\newcommand{\CX}{\breve{\cC}_X}

\newcommand{\CC}[1]{\breve{\cC}_{#1}}

\newcommand{\rsep}[2]{({#1},{#2})}
\newcommand{\lsep}[2]{({#1},{#2})}
\newcommand{\sep}[2]{\{{#1},{#2}\}}

\newcommand{\trsep}[1]{\rsep{#1}{\cK({#1})}}
\newcommand{\tlsep}[1]{\lsep{\cK({#1})}{#1}}
\newcommand{\tsep}[1]{\sep{#1}{\cK({#1})}}

\renewcommand{\subset}{\subseteq}

\def\comp{com\-pac\-ti\-fi\-ca\-tion}

\def\SC{Stone-Čech}

\def\TFAD{Let $G$ be any connected graph and let $U\subset V(G)$ be any vertex set. Then the following assertions are complementary:}

\newcommand{\at}{attached to }

\newcommand{ \N } { \mathbb{N} }

\newcommand{\dblue}[1]{\textcolor{darkishBlue}{#1}}

\makeatletter

\def\calCommandfactory#1{%
   \expandafter\def\csname c#1\endcsname{\mathcal{#1}}}
\def\frakCommandfactory#1{%
   \expandafter\def\csname frak#1\endcsname{\mathfrak{#1}}}

\newcounter{ctr}
\loop
  \stepcounter{ctr}
  \edef\X{\@Alph\c@ctr}
  \expandafter\calCommandfactory\X
  \expandafter\frakCommandfactory\X
  \edef\Y{\@alph\c@ctr}
  \expandafter\frakCommandfactory\Y
\ifnum\thectr<26
\repeat

\renewcommand{\cC}{\mathscr{C}}

\renewcommand{\cK}{\mathscr{K}}

\renewcommand{\cF}{\mathscr{F}}

\setenumerate{label={\normalfont (\roman*)},itemsep=0pt}

\def\lowfwd #1#2#3{{\mathop{\kern0pt #1}\limits^{\kern#2pt\raise.#3ex
\vbox to 0pt{\hbox{$\scriptscriptstyle\rightarrow$}\vss}}}}
\def\lowbkwd #1#2#3{{\mathop{\kern0pt #1}\limits^{\kern#2pt\raise.#3ex
\vbox to 0pt{\hbox{$\scriptscriptstyle\leftarrow$}\vss}}}}
\def\fwd #1#2{{\lowfwd{#1}{#2}{15}}}
\def\Sinf{S_{\aleph_0}}

\def\vS{{\hskip-1pt{\fwd S3}\hskip-1pt}}
\def\vSinf{\vS_{\mkern-.85\thinmuskip\aleph_0}}

\def\vE{{\hskip-1pt{\fwd{E}{3.5}}\hskip-1pt}}
\def\vF{{\hskip-1pt{\fwd{F}{3.5}}\hskip-1pt}}
\def\ve{\kern-1.5pt\lowfwd e{1.5}2\kern-1pt}
\def\ev{\kern-1pt\lowbkwd e{0.5}2\kern-1pt}
\def\vf{\kern-2pt\lowfwd f{2.5}2\kern-1pt}

\def\vT{\lowfwd T{0.3}1}
\def\vs{\lowfwd s{1.5}1}
\def\sv{\lowbkwd s{0}1}

\newtheorem{theorem}{Theorem}[section] 
\newtheorem{corollary}[theorem]{Corollary}
\newtheorem{lemma}[theorem]{Lemma}

\newtheorem{mainresult}{Theorem} 

\newtheorem*{NNtheorem}{Theorem}

\newtheorem{innercustomthm}{Theorem}

\newenvironment{customthm}[1]
  {\renewcommand\theinnercustomthm{#1}\innercustomthm}
  {\endinnercustomthm}

\theoremstyle{definition}

\newtheorem{definition}[theorem]{Definition}

\theoremstyle{remark}

\def\principal{principal}
\def\admissable{admissable}

\begin{document}

\title[Duality theorems for stars and combs IV: Undominating stars]{Duality theorems for stars and combs\\
IV: Undominating stars}

\author{Carl Bürger}
\author{Jan Kurkofka}
\address{University of Hamburg, Department of Mathematics, Bundesstraße 55 (Geomatikum), 20146 Hamburg, Germany}
\email{carl.buerger@uni-hamburg.de, jan.kurkofka@uni-hamburg.de}


\@namedef{subjclassname@2020}{\textup{2020} Mathematics Subject Classification}

\begin{abstract}
In a series of four papers we determine structures whose existence is dual, in the sense of complementary, to the existence of stars or combs.
In the first paper of our series we determined structures that are complementary to arbitrary stars or combs.
Stars and combs can be combined, positively as well as negatively.
In the second and third paper of our series we provided duality theorems for all but one of the possible combinations.

In this fourth and final paper of our series, we complete our solution to the problem of finding complementary structures for stars, combs, and their combinations, by presenting duality theorems for the missing piece: for undominating stars.
Our duality theorems are phrased in terms of end-compactified subgraphs, tree-decompositions and tangle-distinguishing separators.
\end{abstract}
\vspace*{-1.14cm} 
\maketitle

\vspace*{-.75cm}

\section{Introduction}

\newcommand{\sls}{\mkern-.5\thinmuskip}
\newcommand{\srs}{\mkern+.25\thinmuskip}

\newcommand{\cor}{\text{\textnormal{cor}}}

\def\vr{\kern-1.5pt\lowfwd r{1.5}2\kern-1pt}
\def\rv{\kern-1pt\lowbkwd r{0.5}2\kern-1pt}
\def\vt{\kern-1.5pt\lowfwd t{1.5}2\kern-1pt}
\def\tv{\kern-1pt\lowbkwd t{0.5}2\kern-1pt}
\def\vD{\lowfwd D{0.3}1}
\def\vK{\lowfwd K{0.3}1}
\newcommand{\order}[1]{\vert{#1}\vert}

\newcommand{\tle}{\lesssim}
\newcommand{\onion}{\pi}
\newcommand{\Onion}{parliament}
\newcommand{\Onionring}{onion ring}
\newcommand{\oniontreeset}{\mathrm{bo}}
\newcommand{\Oniontreeset}{bi\-orien\-ted onion}


\noindent Two properties of infinite graphs are \emph{complementary} in a class of infinite graphs if they partition the class.
In a series of four papers we determine structures whose existence is complementary to the existence of two substructures that are particularly fundamental to the study of connectedness in infinite graphs: stars and combs.
See~\cite{StarComb1StarsAndCombs} for a comprehensive introduction, and a brief overview of results, for the entire series of four papers (\cite{StarComb1StarsAndCombs,StarComb2TheDominatedComb,StarComb3TheUndominatedComb} and this paper).

In the first paper~\cite{StarComb1StarsAndCombs} of this series we found structures whose existence is complementary to the existence of a star or a comb attached to a given set $U$ of vertices, and two types of these structures turned out to be relevant for both stars and combs: normal trees and tree-decompositions.
A \emph{comb} is the union of a ray $R$ (the comb's \emph{spine}) with infinitely many disjoint finite paths, possibly trivial, that have precisely their first vertex on~$R$. 
The last vertices of those paths are the \emph{teeth} of this comb.
Given a vertex set $U$, a \emph{comb attached to} $U$ is a comb with all its teeth in $U$, and a \emph{star attached to} $U$ is a subdivided infinite star with all its leaves in $U$.
Then the set of teeth is the \emph{attachment set} of the comb, and the set of leaves is the \emph{attachment set} of the star.
Given a graph $G$, a rooted tree $T\subset G$ is \emph{normal} in $G$ if the endvertices of every $T$-path in $G$ are comparable in the tree-order of $T$, cf.~\cite{DiestelBook5}.
For the definition of tree-decompositions see~\cite{DiestelBook5}.

As stars and combs can interact with each other, this is not the end of the story.
For example, a given vertex set $U$ might be connected in a graph $G$ by both a star and a comb, even with infinitely intersecting sets of leaves and teeth. 
To formalise this, let us say that a subdivided star $S$ \emph{dominates} a comb $C$ if infinitely many of the leaves of $S$ are also teeth of $C$.
A \emph{dominating star} in a graph~$G$ then is a subdivided star $S\subset G$ that dominates some comb $C\subset G$; and a \emph{dominated comb} in $G$ is a comb $C\subset G$ that is dominated by some subdivided star $S\subset G$.
Thus, a star $S\subset G$ is undominating in $G$ if it is not dominating in $G$; and a comb $C\subset G$ is undominated in $G$ if it is not dominated in $G$.

In the second paper~\cite{StarComb2TheDominatedComb} of our series we determined structures whose existence is complementary to the existence of dominating stars or dominated combs.
Like for arbitrary stars and combs, our duality theorems for dominating stars and dominated combs are phrased in terms of normal trees and tree-decompositions.

In the third paper~\cite{StarComb3TheUndominatedComb} of the series we determined structures whose existence is complementary to the existence of undominated combs.
Our investigations showed that the types of complementary structures for undominated combs are quite different compared to those for stars, combs, dominating stars and dominated combs.
On the one hand, normal trees are too strong to serve as complementary structures, which is why we considered more general subgraphs instead.
Tree-decompositions on the other hand are dynamic enough to allow for duality theorems, even in terms of star-decompositions---which are too strong to serve as complementary structures for stars, combs, dominating stars or dominated combs.

Among all the combinations of stars and combs, there is only one combination that we have yet to consider: undominating stars.
Here, in the fourth and final paper of the series, we determine structures whose existence is complementary to the existence of
undominating stars.
The types of complementary structures for undominating stars differ from those for stars, combs, dominating stars and dominated combs---surprisingly in the same way the types of complementary structures for undominated combs differ from them.

To begin, normal trees are too strong to serve as complementary structures for undominating stars: 
if $G$ is an uncountable complete graph and $U=V(G)$, then $G$ contains no undominating star \at $U$ but $G$ has no normal spanning tree.
However, if $G$ contains no undominating star \at $U$ and $U$ happens to be contained in a normal tree $T\subset G$, 
then the down-closure of $U$ in $T$ forms a locally finite subtree $H$.
In this situation $H$ witnesses that $U$ is \emph{tough} in $G$ in that
only finitely many components meet $U$ whenever finitely many vertices are deleted from~$G$.
This property gives a candidate for a subgraph that might serve as a complementary structure, even when $U$ is not contained in a normal tree.
Call a graph $G$ \emph{tough} if its vertex set is tough in $G$, i.e., if deleting finitely many vertices from $G$ always results in only finitely many components.
It is well known that the tough graphs are precisely the graphs that are compactified by their ends, cf.~\cite{VTopComp}.
Our first duality theorem for undominating stars is formulated in terms of tough subgraphs:

\begin{customthm}{\ref{thm: toughDual}}
\TFAD
\begin{enumerate}
    \item $G$ contains an \dblue{undominating star} \at $U$;
    \item there is a \dblue{tough subgraph} $H\subset G$ that contains $U$.
\end{enumerate}
\end{customthm}

As our second duality theorem for undominating stars, we also find star-decom\-po\-si\-tions that are complementary to undominating stars:

\begin{customthm}{\ref{StarDecomposition}}
\TFAD
\begin{enumerate}
    \item $G$ contains an \dblue{undominating star} \at $U$;
    \item $G$ has a tame \dblue{star-decomposition} such that $U$ is contained in the central part and every critical vertex set of $G$ lives in a leaf's part.
\end{enumerate}
\end{customthm}
\noindent Here, a finite vertex set $X\subset V(G)$ is \emph{critical} if infinitely many of the components of $G-X$ have their neighbourhood precisely equal to $X$.
Critical vertex sets were introduced in~\cite{EndsTanglesCrit}.
As tangle-distinguishing separators, they have a surprising background involving the \SC\ \comp\ of $G$, Robertson and Seymour's tangles from their graph-minor series, and Diestel's tangle \comp , cf.~\cite{StoneCechTangles,GMX,EndsAndTangles}.
For the definitions of `tame' and `live', see Section~\ref{section:undomStarStarDecomposition}.
Tame tree-decompositions have finite adhesion sets.

While the wordings of our two duality theorems for undominating stars are similar to those of the duality theorems for undominated combs, their proofs are~not.
In fact, a whole new strategy is needed to prove these two theorems.
The starting point of our strategy will be a very recent generalisation~\cite{DistinguishUltrafilterTangles} of Robertson and Seymour's tree-of-tangles theorem from their graph-minor series~\cite{GMX}.

This paper is organised as follows.
Section~\ref{section:undomStarToughSubgraph} establishes our duality theorem for undominating stars in terms of end-compactified subgraphs.
Section~\ref{section:undomStarStarDecomposition} proves our duality theorem for undominating stars in terms of star-decompositions. 
In Section~\ref{sec:summary} we summarise the duality theorems of the complete series.

Throughout this paper, $G=(V,E)$ is an arbitrary  graph.
We use the graph theoretic notation of Diestel's book~\cite{DiestelBook5}, and we assume familiarity with the tools and terminology described in the first paper of this series~\cite[Section~2]{StarComb1StarsAndCombs}.
For definitions and basic properties regarding separation systems refer to~\cite{AbstractSepSys}.

\section{Tough subgraphs}\label{section:undomStarToughSubgraph}

\noindent In this section, we prove our duality theorem for undominating stars in terms of tough subgraphs: 
\begin{mainresult}\label{thm: toughDual}
\TFAD
\begin{enumerate}
    \item $G$ contains an undominating star \at $U$;
    \item $G$ has a tough subgraph that contains $U$.
\end{enumerate}
\end{mainresult}

\noindent We remark that the tough graphs are precisely the graphs that are compactified by their ends, see~\cite{VTopComp}.

We prove that (i) and (ii) are complementary by proving that both $\neg$(i) and (ii) are equivalent to the assertion that $U$ is tough in $G$.
That $\neg$(i) is equivalent to $U$ being tough in $G$ will be shown in Lemma~\ref{lem: undominating star and tough}, and that (ii) is equivalent to $U$ being tough in $G$ will be shown in Theorem~\ref{thm: tough vertex set tough subgraph}.
It will be convenient to make this detour because $U$ being tough in $G$ is easier to work with than $G$ not containing an undominating star \at $U$.

\begin{lemma}\label{lem: undominating star and tough}
A set $U$ of vertices of a connected graph $G$ is tough in $G$ if and only if $G$ contains no undominating star \at $U$.
\end{lemma}

\begin{theorem}\label{thm: tough vertex set tough subgraph}
A set $U$ of vertices of a graph $G$ is tough in $G$ if and only if $G$ has a tough subgraph that contains $U$.
\end{theorem}

\begin{proof}[Proof of Theorem~\ref{thm: toughDual}]
Combine Lemma~\ref{lem: undominating star and tough} and Theorem~\ref{thm: tough vertex set tough subgraph} above.
\end{proof}

\noindent While the  proof of Theorem~\ref{thm: tough vertex set tough subgraph} takes the rest of this section, that of Lemma~\ref{lem: undominating star and tough} is easy and we shall provide it straight away. Recall that a finite set $X$ of vertices of an infinite graph $G$ is \emph{critical} if the collection
\begin{align*}
    \CX:=\{\,C\in\cC_X\mid N(C)=X\,\}
\end{align*}
is infinite, where $\cC_X$ is the collection of all components of $G-X$. 
A critical vertex set $X$ of $G$ lies \emph{in the closure} of $M$, where $M$ is either a subgraph of $G$ or a set of vertices of $G$, if infinitely many components in $\CX$ meet $M$.

\begin{proof}[Proof of Lemma~\ref{lem: undominating star and tough}]
If $U$ is tough in $G$ then no critical vertex set of $G$ lies in the closure of~$U$. 
We know by \cite[Lemma~2.9]{StarComb1StarsAndCombs} that every infinite set of vertices in a connected graph has an end or a critical vertex set in its closure.
Therefore, every infinite subset $U'\subset U$ has an end of $G$ in its closure and, in particular, there is always a comb in $G$ \at $U'$.
Thus, every star in $G$ \at $U$ must be dominating.

Conversely, if $U$ is not tough in $G$, then there is a finite vertex set $X\subset V(G)$ such that some infinitely many components of $G-X$ meet $U$.
Then infinitely many of these components send an edge to the same vertex $x\in X$ by the pigeonhole principle. This allows us to make $x$ the centre of a star $S$ \at $U$ by taking $x$--$U$ paths in $G[x+C]$, one for each of the infinitely many components~$C$ that meet $U$ and have $x$ in their neighbourhood.
Now $X$ obstructs the existence of a comb that has infinitely many teeth that are also leaves of $S$, and so $S$ must be undominating.
\end{proof}

Before we turn to the proof of Theorem~\ref{thm: tough vertex set tough subgraph}, we summarise a few elementary properties that are complementary to containing an undominating star \at a given vertex set $U$:

\begin{lemma}\label{lem: list}
Let $G$ be any connected graph, let $U\subseteq V(G)$ be any vertex set and let \emph{($\ast$)} be the statement that $G$ contains an undominating star \at ~$U$. 
Then the following assertions are complementary to \emph{($\ast$):}
\begin{enumerate}
    \item $U$ is tough in $G$;
    \item $G$ has no critical vertex set that lies in the closure of $U$;
    \item $U$ is compactified by the ends of $G$ that lie in the closure of $U$.
\end{enumerate}
If $U$ is normally spanned in $G$, then the following assertion is complementary to \emph{$(\ast)$} as well:
\begin{enumerate}[resume]
    \item $G$ contains a locally finite normal tree that contains $U$ cofinally.
\end{enumerate}
\end{lemma}
\begin{proof}

By Lemma~\ref{lem: undominating star and tough} we have that (i) is complementary to ($\ast$).
The assertions (i) and (ii) are equivalent by the pigeonhole principle, and hence (ii) is complementary to ($\ast$) as well.
Property (iii) is in turn equivalent to (ii) because every graph is compactified by its ends and critical vertex sets in a compactification $\minG=G\cup\Omega(G)\cup\crit(G)$ (see~\cite{EndsTanglesCrit} for definitions): For (ii)$\to$(iii) note that the closure $\overline{U}=U\cup\Abs{U}$ of $U$ in $\minG$ is the desired compactification, and for $\neg$(ii)$\,{\to}\,{\neg}$(iii) note that for every critical vertex set $X$ in the closure of $U$ the infinitely many components of $G-X$ meeting $U$ give rise to an open cover of $U\cup\Abs{U}$  in $\minG$ that has no finite subcover.
That (iv) is complementary to ($\ast$) has already been discussed in the introduction.
\end{proof}

Now we turn to the proof of Theorem~\ref{thm: tough vertex set tough subgraph}. 
If a graph $G$ has a tough subgraph containing some vertex set $U$, then clearly $U$ is tough in $G$.
The reverse implication, which states that that for every vertex set $U$ that is tough in $G$ the graph $G$ contains a tough subgraph containing $U$, is harder to show and needs some preparation.

If $U$ is tough in $G$, then no critical vertex set of $G$ lies in the closure of $U$,
that is, for every critical vertex set $X$ of $G$ only finitely many components in $\CX$ meet~$U$. 
The collection $\cC(X)$ of these finitely many components gives rise to a separation $\lsep{\CX\setminus\cC(X)}{X}=(A_X,B_X)$ that we think of as pointing towards $B_X$.
As $U\subseteq B_X$ for all critical vertex sets $X$, all the separations $(A_X,B_X)$ point towards the tough vertex set~$U$.
Hence we have a candidate for a tough subgraph: the intersection $\bigcap\,\{\,G[B_X]\mid X\in\crit(G)\,\}$.
This candidate contains $U$ because $U$ is contained in all $G[B_X]$, but it can happen that our candidate is a non-tough induced $\overline{K^{\aleph_0}}\subset G$ with vertex set~$U$, as the following example shows.

For every $n\in \N$ let $A_n$ be some countably infinite set, such that $A_n$ is disjoint from every $A_m$ with $m\neq n$ and also disjoint from $\N$. Let $G$ be the graph on $\N\cup\bigcup_{n\in\N}A_n$ where every vertex in $A_n$ is joined completely to $\{0,\dots, n\}$.
Then the critical vertex sets are precisely the vertex sets of the form $\{0,\ldots,n\}$.
For every critical vertex set $X=\{0,\ldots,n\}$ the collection of components $\CX$ consists of the singletons in $A_n$ and the component of $G-X$ that contains $\N\setminus X$.
Therefore, if we set $U=\N$, then $G[B_X]=G-A_n$, and our candidate $\bigcap_X G[B_X]$ turns out to be $G[\N]=\overline{K^{\aleph_0}}$.

Although our approach in its naive form fails, this is not the end of it. We will stick to the idea but perform the construction in a more sophisticated way. For this we shall need the following notation and two structural results from \cite{DistinguishUltrafilterTangles} for critical vertex sets in graphs, Theorems~\ref{principalAdmissableFunctionExists} and~\ref{TreeSetAPC} below. Essentially, these two theorems together will reveal that the separations $(A_X,B_X)$ with $X$ critical in $G$ can be slightly modified to form a tree set.

A \emph{tree set} is a nested separation system that has neither trivial elements nor degenerate elements, cf.~\cite{RhdTreeSets}.
When $(\vS,{\le},{}^\ast)$ is a tree set, we also call $\vS$ and $S$ tree sets.
In our setting, we shall not have to worry about trivial or degenerate separations too much.
Indeed, usually our nested sets of separations will consist of separations $(A,B)$ of a graph with neither $A\setminus B$ nor $B\setminus A$ empty, and these sets are known to form \emph{regular} tree sets: tree sets that do not contain small elements.

Let $S$ be any tree-set consisting of finite-order separations of $G$. A \emph{part} of $S$ is a vertex set of the form $\bigcap\,\{\,B\mid (A,B)\in O\,\}$ where $O$ is a consistent orientation of $S$. Thus, if $O$ is any consistent orientation of $S$, then it defines a part, which in turn induces a subgraph of $G$. The graph obtained from this subgraph by adding an edge $xy$ whenever $x$ and $y$ are two vertices of the part that lie together in the separator of some separation in $O$ is called the \emph{torso} of $O$ (or of the part, if $O$ is clear from context).
Thus, torsos usually will not be subgraphs of~$G$.
We need the following standard lemma:

\begin{lemma}[{\cite[Corollary~2.11]{DistinguishUltrafilterTangles}}]\label{restrictconnectedsets}
Let $G$ be any graph and let $W\subset V(G)$ be any connected vertex set.
If $B$ is a part of a tree set of separations of $G$, then $W\cap B$ is connected in the torso of $B$.
\end{lemma}

Given a collection $\cY$ of (in this paper usually finite) vertex sets of $G$ we say that a vertex set $X$ of $G$ is $\cY$-\emph{\principal } if $X$ meets for every $Y\in\cY$ at most one component of $G-Y$.
And we say that $\cY$ is \emph{\principal } if all its elements are $\cY$-\principal .

If $X\subset V(G)$ meets precisely one component of $G-Y$ for some $Y\subset V(G)$, then we denote this component by $C_Y(X)$.

Every critical vertex set of a graph is $\cX$-\principal : since every two vertices in a critical vertex set $X$ are linked by infinitely many independent paths (these exist as $\CX$ is infinite), no two vertices in $X$ are separated by a finite vertex set.

\begin{definition}[{\cite[Definition~5.9]{DistinguishUltrafilterTangles}}]
Suppose that $\cY$ is a \principal\ collection of vertex sets of a graph $G$.
A function that assigns to every $X\in\cY$ a subset $\cK(X)\subset\CX$ is called \emph{\admissable } for $\cY$ if for every two $X,Y\in\cY$ that are incomparable as sets we have either $C_X(Y)\notin\cK(X)$ or $C_Y(X)\notin\cK(Y)$.
If additionally \mbox{$\vert\,\CX\setminus\cK(X)\,\vert\le 1$} for all $X\in\cY$, then $\cK$ is \emph{strongly} \admissable\ for $\cY$.
\end{definition}

\begin{theorem}[{\cite[Theorem~5.10]{DistinguishUltrafilterTangles}}]\label{principalAdmissableFunctionExists}
For every \principal\ collection of vertex sets of a connected graph there is a strongly \admissable\ function.
\end{theorem}

\begin{theorem}[{\cite[Theorem~5.11]{DistinguishUltrafilterTangles}}]\label{TreeSetAPC}
Let $G$ be any connected graph, let $\cY$ be any \principal\ collection of vertex sets of $G$ and let $\cK$ be any \admissable\ function for~$\cY$.
Then for every distinct two $X,Y\in\cY$, after possibly swapping $X$ and $Y$,
\begin{align*}
    \text{either }\tlsep{X}\le\trsep{Y}\text{ or }\tlsep{X}\le\lsep{C_Y(X)}{Y}\le\tlsep{Y}.\
\end{align*}
In particular, if $\emptyset\subsetneq\cK(X)\subsetneq\cC_X$ for all $X\in\cY$, then the separations $\tsep{X}$ form a regular tree set for which the separations $\tlsep{X}$ form a consistent orientation.
\end{theorem}

Suppose now that $\cY$ is a \principal\ collection of vertex sets of a graph $G$ and that $\cK$ is an \admissable\ function for $\cY$ satisfying $\emptyset\subsetneq\cK(X)\subsetneq\cC_X$ for all $X\in\cY$.
If $T$ is the regular tree set $\{\,\tsep{X}\mid X\in\cY\,\}$ provided by Theorem~\ref{TreeSetAPC}, then we call $T$ a \emph{\principal } tree set of $G$.
By a slight abuse of notation, we also call the triple $(T,\cY,\cK)$ a \principal\ tree set.
In this context, we write $O_\cK$ for the consistent orientation $\{\,\tlsep{X}\mid X\in\cY\,\}$ of~$T$.

\begin{corollary}\label{cor: no dominating star implies principal tree set}
Let $G$ be any connected graph and let $U\subseteq V(G)$ be any vertex set. 
If $U$ is tough in $G$, then there is a principal tree set $(T,\crit(G),\cK)$ of $G$ satisfying the following two conditions:
\begin{enumerate}
    \item no element of $\cK(X)$ meets $U$ for any critical vertex set $X$;
    \item $\cK(X)$ is a cofinite subset of $\CX$ for every critical vertex set $X$.
\end{enumerate} 
\end{corollary}

\begin{proof}
As $U$ is tough in $G$, for every critical vertex set $X$ of $G$ only finitely many components in $\CX$ meet $U$;
we write $\cF_X$ for this finite collection. Theorem~\ref{principalAdmissableFunctionExists} yields a strongly \admissable\ function $\cK$ for the collection $\crit(G)$ of all the critical vertex sets of $G$.
We alter this function by removing $\cF_X$ from $\cK(X)$ for all $X$.
Then $\cK$ is still \admissable\ for $\crit(G)$, and $\cK(X)$ is a cofinite subcollection of $\CX\setminus\cF_X$ for all $X$.
Now Theorem~\ref{TreeSetAPC} says that the separations $\tsep{X}$ with $X$ critical form a tree set, and that the oriented separations $\tlsep{X}$ form a consistent orientation of this tree~set.
\end{proof}

\begin{proof}[{Proof of Theorem~\ref{thm: tough vertex set tough subgraph}}]
If $H$ is a tough subgraph of $G$ covering $U$, then $U$ is tough in $H$; in particular, $U$ is tough in $G$. Conversely, we need to show that for every vertex set $U\subseteq V(G)$ that is tough in $G$ there is a tough subgraph of $G$ containing~$U$. 
By Corollary~\ref{cor: no dominating star implies principal tree set} we find a principal tree set $(T,\crit(G),\cK)$ so that, for every critical vertex set $X$, no element of $\cK(X)$ meets $U$ and $\cK(X)$ is a cofinite subset of~$\CX$.
We write $B$ for the part of $T$ that is defined by $O_\cK$. Note that $U$ is included in $B$.

First we claim that the torso of the part $B$ is tough.
To see this, consider any finite vertex set $X\subset B$.
Only finitely many components of $G-X$ meet $B$: indeed, if infinitely many components of $G-X$ meet $B$, then by the pigeonhole principle we deduce that a subset $X'$ of $X$ is critical in $G$ with infinitely many components in $\CC{X'}$ meeting $B$.
But then $\bigcup\cK(X')$ must meet $B$, contradicting that $B$ is the part of $T$ that is defined by $O_\cK=\{\, \tlsep{X}\mid X\in \crit(G)\,\}$.
Thus $G-X$ has only finitely many components meeting $B$.
By Lemma~\ref{restrictconnectedsets} each of these components induces a component of the torso minus $X$, and so deleting $X$ from the torso results in at most finitely many components.

The tough torso of the part $B$, however, usually is not a subgraph of $G$. 
And the part $B$ usually will not induce a tough subgraph of $G$.
That is why as our next step, we construct a subgraph $H$ of $G$ that imitates the torso of $B$ to inherit its toughness.
More precisely, we obtain $H$ from  $G[B]$ by adding a subgraph $L$ of $G$ that has the following three properties:

\begin{enumerate}[label=(L\arabic*)]
    \item\label{LlocFin} Every vertex of $L-B$ has finite degree in $L$.
    \item\label{LfinYfinComps} For every finite $X\subset B$ only finitely many components of $L-X$ avoid $B$. 
    \item\label{LemulatesTorso} If $x$ and $y$ are distinct vertices in $B$ that lie together in a critical vertex set of $G$, then $L$ contains a $B$-path between $x$ and $y$.
\end{enumerate}

Before we begin the construction of $L$, let us verify that any $L$ satisfying these three properties really gives rise to a tough subgraph $H=G[B]\cup L$.
For this, consider any finite vertex set $X\subset V(H)$.
By \ref{LlocFin} every vertex of $H-B$ has finite degree in $H$, and hence deleting it produces only finitely many new components.
Therefore we may assume that $X$ is included in $B$ entirely.
Every component of $H-X$ avoiding $B$ is a component of $L-X$ avoiding $B$, and there are only finitely many such components by \ref{LfinYfinComps}.
Hence it remains to show that there are only finitely many components of $H-X$ that meet $B$.
We already know that the torso of $B$ is tough, so deleting $X$ from it results in at most finitely many components.
Then property \ref{LemulatesTorso} ensures that each of these finitely many components has its vertex set included in a component of $H-X$.
And hence there can only be finitely many components of $H-X$ that meet $B$.

Finally, we construct a subgraph $L\subseteq G$ satisfying the three properties (L1), (L2) and (L3).
Choose $(\,\{x_\alpha,y_\alpha\}\,)_{\alpha<\kappa}$ to be a transfinite enumeration of the collection of all unordered pairs $\{x,y\}$ where $x$ and $y$ are distinct vertices in $B$ that lie together in a critical vertex set of $G$.
Then we recursively construct $L$ as a union $L=\bigcup_{\alpha<\kappa}P_\alpha$ where at step $\alpha$ we choose $P_\alpha$ 
from among all $B$-paths $P$ in $G$ between $x_\alpha$ and $y_\alpha$ 
so as to minimize the number $\vert E(P)\setminus E(\bigcup_{\xi<\alpha}P_{\xi})\vert$ of new edges.
(There is a $B$-path in $G$ between $x_\alpha$ and $y_\alpha$ since $x_\alpha$ and $y_\alpha$ lie together in some critical vertex set $X$ of $G$ and $\cK(X)\subset\CX$ is non-empty.) 

We verify that our construction yields an $L$ satisfying \ref{LlocFin}, \ref{LfinYfinComps} and \ref{LemulatesTorso}.

\ref{LlocFin}.
For this, fix any vertex $\ell\in L-B$.
It suffices to show that the edges of $L$ at $\ell$ simultaneously extend to an $\ell$--$B$ fan in $L$.
To see that this really suffices, use that $\ell$ is not contained in $B$ to find some critical vertex set $X$ of $G$ with $\ell\in\bigcup\cK(X)$.
Then the $\ell$--$B$ fan at $\ell$ extending the edges of $L$ at $\ell$ must have all its $\ell$--$B$ paths pass through the finite $X$, and so there can be only finitely many such paths, meaning that $\ell$ has finite degree in $L$.

Now to find the $\ell$--$B$ fan we proceed as follows.
For every edge $e$ of $L$ at $\ell$ we write $\alpha(e)$ for the minimal ordinal $\alpha$ with $e\in E(P_\alpha)$.
Then we write $P_e$ for $P_{\alpha(e)}$, and we write $Q_e$ for the $\ell$--$B$ subpath of $P_e$ containing $e$.
The paths $Q_e$ form an $\ell$--$B$ fan, as we verify now.
For this, we show that, if $e\neq e'$ are two distinct edges of $L$ at $\ell$, then $Q_e$ and $Q_{e'}$ meet precisely in $\ell$.
Let $e$ and $e'$ be given.
We abbreviate $\alpha(e)=\alpha$ and $\alpha(e')=\alpha'$.
If $\alpha=\alpha'$ then $Q_e\cup Q_{e'}=P_\alpha$ and we are done.
Otherwise $\alpha<\alpha'$, say.
Then we assume for a contradiction that $\mathring{\ell}Q_{e'}$ does meet $\mathring{\ell}Q_e$.
Without loss of generality we may assume that $Q_{e'}$ starts in $\ell$ and ends in $y_{\alpha'}$.
We let $t$ be the last vertex of $Q_{e'}$ in $\mathring{\ell}Q_e$.
But then the graph $x_{\alpha'} P_{e'}\ell\cup \ell Q_e t P_{e'} y_{\alpha'}$ is connected and meets $B$ precisely in the two vertices $x_{\alpha'}$ and $y_{\alpha'}$.
Consequently, it contains a $B$-path $P$ between $x_{\alpha'}$ and $y_{\alpha'}$.
But then $P$ avoids the edge $e'$, so the inclusion $E(P)\setminus E(\bigcup_{\xi<\alpha'}P_{\xi})\subset E(P_{e'})\setminus E(\bigcup_{\xi<\alpha'}P_{\xi})$ must be proper.
Therefore, $P$ contradicts the choice of $P_{\alpha'}$ as desired.

\ref{LfinYfinComps}. 
For this, fix any finite vertex set $X\subset B$.
Let $\cC$ be the set consisting of all the components of $L-X$ that avoid $B$.
And let $F$ consist of all the edges inside components from $\cC$ and all the edges of $L$ between components from $\cC$ and $X$, i.e., $F=E(\bigcup\cC)\cup E_L(\bigcup\cC,X)$. 
As every component from $\cC$ meets some edge from $F$ it suffices to show that $F$ is finite, a fact that we verify as follows. 
Every edge in $F$ lies on a path $P_\alpha$, and since $P_\alpha$ is a $B$-path between $x_\alpha$ and $y_\alpha$ we deduce $\{x_\alpha,y_\alpha\}\in [X]^2$.
Thus the finite edge sets of the paths $P_\alpha$ with $\{x_\alpha,y_\alpha\}\in [X]^2$ cover $F$.
Since $X$ is finite so is $[X]^2$, and hence there are only finitely many such paths, meaning that $F$ is finite.

\ref{LemulatesTorso}. This property holds by construction.

As (L1), (L2) and (L3) are now verified we conclude that $L$ is as desired, which completes the proof of our first main result.
\end{proof}

\section{Star-decompositions}\label{section:undomStarStarDecomposition}

\noindent In this section we prove our second main result, a duality theorem for undominating stars in terms of star-decompositions, Theorem~\ref{StarDecomposition} below.

Before we state the theorem, let us recall the following definitions from~\cite[Section~3.5]{StarComb1StarsAndCombs}.
A finite-order separation $\sep{X}{\cC}$ of a graph $G$ is \emph{tame} if for no $Y\subset X$ both $\cC$ and $\cC_X\setminus\cC$ contain infinitely many components whose neighbourhoods are precisely equal to~$Y$.
The tame separations of $G$ are precisely the finite-order separations of $G$ that respect the critical vertex sets:
\begin{lemma}[{\cite[Lemma~3.15]{StarComb1StarsAndCombs}}]\label{lemma: char tame via crit vx sets}
A finite-order separation $\{A,B\}$ of a graph $G$ is tame if and only if every critical vertex set $X$ of $G$ together with all but finitely many components from $\CX$ is contained in one side of $\{A,B\}$.
\end{lemma}
An $\Sinf$-tree $(T,\alpha)$ is \emph{tame} if all the separations in the image of $\alpha$ are tame.
As a consequence of Lemma~\ref{lemma: char tame via crit vx sets}, if $X$ is a critical vertex set of $G$ and $(T,\alpha)$ is a tame $\Sinf$-tree, then $X$ induces a consistent orientation of the image of $\alpha$ by orienting every tame finite-order separation $\{A,B\}$ towards the side that contains $X$ and all but finitely many of the components from~$\CX$.
This consistent orientation, via $\alpha$, also induces a consistent orientation of $\vE(T)$.
Then, just like for ends, the critical vertex set $X$ either \emph{lives} at a unique node $t\in T$ or \emph{corresponds} to a unique end of~$T$.
As usual, these definitions for $\Sinf$-trees carry over to tree-decompositions.

\begin{mainresult}\label{StarDecomposition}
\TFAD
\begin{enumerate}
    \item $G$ contains an undominating star \at $U$;
    \item $G$ has a tame star-decomposition such that $U$ is contained in the central part and every critical vertex set of $G$ lives in a leaf's part.
\end{enumerate}
\end{mainresult}

The proof of this theorem is organised as follows.
First, we state without proof a technical theorem, Theorem~\ref{TechnicalStarDecomposition} below, and then show how it implies our main result, Theorem~\ref{StarDecomposition} above.
In a last step we prepare and provide the proof of the technical theorem.

Note that the part of a star $\sigma$ of separations of a graph $G$ is \mbox{$\bigcap\,\{\,B\mid (A,B)\in\sigma\,\}$}.
Given two oriented separations $\vs_{\! 1},\vs_{\! 2}$ of $G$ we write $\vs_{\! 1}\tle\vs_{\! 2}$ if either \mbox{$\vs_{\! 1}\le\vs_{\! 2}$} or there is a component $C\in\cC$ for $\lsep{\cC}{X}=\vs_{\! 1}$ such that \mbox{$\lsep{\cC\setminus\{C\}}{X}\le\vs_{\! 2}$}.
Here is the technical theorem:

\begin{theorem}\label{TechnicalStarDecomposition}
Let $G$ be any graph, and let $(T, \cY,\cK)$ be any \principal\ tree set so that $O_\cK$ defines an infinite part.
Then $G$ admits a star $\sigma$ of finite-order separations such that the following two conditions hold:
\begin{enumerate}
    \item the part defined by $O_\cK$ is included in the part of $\sigma$;
    \item for every $\vs\in O_\cK$  there is some $\vr\in \sigma$ with $\vs\tle\vr$.
\end{enumerate}
\end{theorem}

\noindent The technical theorem implies our second main result, Theorem~\ref{StarDecomposition}:

\begin{proof}[Proof of Theorem~\ref{StarDecomposition}]
First, we show that at most one of (i) and (ii) holds.
By Lemma~\ref{lem: list} we know that
if $G$ contains an undominating star \at $U$, then $G$ has a critical vertex $X$ that lies in the closure of $U$.
But then $X$ lives in a leaf's part of the star-decomposition provided by~(ii), and it follows that this part does contain infinitely many vertices from~$U$, contradicting that $U$ is contained in the central part and that the separations of the star-decomposition are finite.

Now, to show that at least one of (i) and (ii) holds, we show $\neg$(i)$\to$(ii).
If $U$ is finite, then the star $\{\,\lsep{\cC_U(Y)}{Y}\mid Y\in 2^U\setminus\{\emptyset\}\,\}$ gives the desired star-decomposition with central part equal to~$U$, where $\cC_U(Y)$ is the collection of all components $C\in\cC_U$ with $N(C)=Y$.
Otherwise $U$ is infinite.
By Lemma~\ref{lem: list} we know that $U$ is tough in $G$.
Then, by Corollary~\ref{cor: no dominating star implies principal tree set}, we find a \principal\ tree set $(T,\crit(G),\cK)$ such that, for every critical vertex set $X$, no element of $\cK(X)$ meets $U$ and the inclusion $\cK(X)\subset\CX$ is cofinite.
We claim that the star provided by Theorem~\ref{TechnicalStarDecomposition} gives a star-decomposition of $G$ meeting the requirements of (ii), a fact that can be verified as follows: 
First, the separations of the form $(\cK(X),X)$ with $X$ critical and $\cK(X)$ a cofinite subset of $\CX$ are tame and thus our star-decomposition is tame. 
Next, by Theorem~\ref{TechnicalStarDecomposition}~(i), we have that $U$ is contained in the central part of the star-decomposition. Finally, by Theorem~\ref{TechnicalStarDecomposition}~(ii), every critical vertex set of $G$ lives in a leaf's part.
\end{proof}

Next, we prepare the proof of our technical theorem, Theorem~\ref{TechnicalStarDecomposition}. 
First, we will need the following theorem by Kneip. 
A chain $\mathcal{C}$ in a given poset is said to have \emph{order-type} $\alpha$ for an ordinal $\alpha$ if $\mathcal{C}$ with the induced linear order is order-isomorphic to $\alpha$. The chain $\mathcal{C}$ is then said to be an $\alpha$-\emph{chain}.

\begin{theorem}[{\cite[Theorem~1]{TreelikeSpaces}}]\label{KneipTreeSets}
A tree set is isomorphic to the edge tree set of a tree if and only if it is regular and contains no $(\omega+1)$-chain.
\end{theorem}

Besides this theorem, we will need the following concept of a corridor from~\cite{DistinguishUltrafilterTangles}.
Suppose that $(\vT,{\le},{}^\ast)$ is a tree set, and that $O$ is a consistent orientation of $\vT$.
A \emph{corridor} of $O$ is an equivalence class of separations in $O$, where two separations $\vs_{\! 1},\vs_{\! 2}\in O$ are considered equivalent if there is $\vr\in O$ with $\vs_{\! 1},\vs_{\! 2}\le\vr$, cf.~\cite[Lemma~7.1 and Definition~7.2]{DistinguishUltrafilterTangles}.
As corridors are consistent partial orientations of tree sets on the one hand, and directed posets on the other hand, they come with a number of useful properties.

The \emph{supremum} $\sup L$ of a set $L$ of oriented separations of a graph is the oriented separation $(A,B)$ with $A=\bigcup\,\{\,C\mid (C,D)\in L\,\}$ and $B=\bigcap\,\{\,D\mid (C,D)\in L\,\}$.

\begin{lemma}\label{supremaAreNested}
Let $T$ be any regular tree set of separations of any graph $G$, let $O$ be any consistent partial orientation of $T$ and let $\gamma$ be any corridor of $O$. 
Then the supremum of $\gamma$ is nested with $\vT$.
\end{lemma}
\begin{proof}
Consider any unoriented separation $r\in T$.
If there is a separation $\vs\in\gamma$ such that $r$ has an orientation $\vr$ with $\vr\le\vs$, then $\vr\le\vs\le\sup\gamma$ as desired.
As $T$ is nested, $r$ has for every separation $\vs\in\gamma$ an orientation $\vr(\vs)$ such that either $\vr(\vs)\le\vs$ or $\vs\le\vr(\vs)$.
By our first observation, we may assume that $\vs\le\vr(\vs)$ for all $\vs\in\gamma$.
It suffices to show that $\vr(\vs_{\! 1})=\vr(\vs_{\! 2})$ for all $\vs_{\! 1},\vs_{\! 2}\in\gamma$, since then $r$ has one orientation that lies above all elements of $\gamma$ and, in particular, above the supremum of $\gamma$.
Given $\vs_{\! 1},\vs_{\! 2}\in\gamma$ consider any $\vs_{\! 3}\in\gamma$ with $\vs_{\! 1},\vs_{\! 2}\le\vs_{\! 3}$.
Then $\vs_{\! 1},\vs_{\! 2}\le\vs_{\! 3}\le\vr(\vs_{\! 3})$.
As $T$ is regular, $\vr(\vs_{\! 3})=\vr(\vs_{\! 1})=\vr(\vs_{\! 2})$ follows.
\end{proof}

\begin{lemma}\label{SupremaOfCorridorsFormStar}
Let $T$ be any tree set of separations of any graph $G$ and let $O$ be any consistent orientation of $T$.
Then the suprema of the corridors of $O$ form a star.
\end{lemma} 
\begin{proof}
We have to show that for every two distinct corridors $\gamma$ and $\delta$ of $O$ the supremum $(A,B)$ of $\gamma$ and the supremum $(C,D)$ of $\delta$ satisfy $(A,B)\le (D,C)$.
Let us write $\gamma=\{\,(A_i,B_i)\mid i\in I\,\}$ and $\delta=\{\,(C_j,D_j)\mid j\in J\,\}$.
As $\gamma$ is distinct from $\delta$ we have $(A_i,B_i)\le (D_j,C_j)$ for all $i\in I$ and $j\in J$.
Hence $(A,B)=(\bigcup_i A_i,\bigcap_i B_i)\le (\bigcap_j D_j,\bigcup_{j} C_j)=(D,C)$.
\end{proof}

\begin{lemma}\label{corridorsAssertMaxima}
Suppose that $T$ is any tree set of separations of any graph $G$, that $O$ is any consistent orientation of $T$, and that $\gamma$ is any corridor of $O$.
Then every finite subset of the separator of the supremum of $\gamma$ is contained in the separator of some separation in~$\gamma$.

In particular, if the order of the separations in $\gamma$ is bounded by some natural number $n$, then the supremum of $\gamma$ has order at most~$n$.
\end{lemma}
\begin{proof}
Let us write $(A,B)$ for the supremum of $\gamma$ and let $Y$ be any finite subset of its separator $X:=A\cap B$.
For every vertex $y\in Y\subset A$ there is separation $(C_y,D_y)\in\gamma$ with $y\in C_y$.
Since $\gamma$ is a corridor we find a separation $(C,D)\in\gamma$ lying above all $(C_y,D_y)$.
Then $Y\subset C$ as $C$ includes all $C_y$, and $Y\subset D$ because $(C,D)\le (A,B)$ gives $Y\subset X\subset B\subset D$.
\end{proof}

Before we start with the proof of Theorem~\ref{TechnicalStarDecomposition} we need two final ingredients: induced separation systems and \Onion s.
If $\vS=(\vS,{\le},{}^\ast)$ is a separation system and $O\subset\vS$ is any subset (usually a partial orientation of~$S$), then $O$ induces a separation system $O\cup O^\ast$ that is a subsystem of $\vS$ with the partial ordering and involution induced by ${\le}$ and ${}^\ast$.
We denote this subsystem by~$\vS[O]$.

Next, we define parliaments.
Suppose that $G$ is any graph, that $\vT=(\vT,{\le},{}^\ast)$ is any regular tree set of finite-order separations of $G$, and that $O$ is any consistent orientation of~$\vT$. 
For every number $n\in\N$ let $O_{\le n}$ be the subset of $O$ formed by the oriented separations in $O$ whose separators have size at most $n$. 
Then, by Lemma~\ref{corridorsAssertMaxima}, every corridor of $O_{\le n}$ has a supremum of order at most $n$, and these suprema form a star for fixed~$n$ (cf.~Lemma~\ref{SupremaOfCorridorsFormStar}) which we denote by~$\onion_n(O)$.
The \emph{\Onion } of $O$, denoted by $\onion(O)$, is the union $\bigcup_{n\in\N}\onion_n(O)$. 
Notably, the \Onion\ of $O$ is a cofinal subset of $O\cup\onion(O)$.
The \Onion\ of $O$ induces a separation system $\vSinf [\onion (O)]$ that is a subsystem of~$\vSinf$ whose separations are all nested with each other.
Furthermore, $\vSinf [\onion (O)]$ and $\vT$ are nested with each other in $\vSinf$ by Lemma~\ref{supremaAreNested}.
Also, the \Onion\ of $O$ is a consistent orientation of $\vSinf [\onion (O)]$ where it defines the same part as $O$ does for $\vT$. 

As one might expect, the inverses of corridors of parliaments have no $\omega$-chains:

\begin{lemma}\label{inversesCorridorsParliaments}
Let $G$ be any graph, let $\vT$ be any regular tree set of finite-order separations of $G$, and let $O$ be any consistent orientation of $\vT$.
Then the inverse $\gamma^\ast$ of any corridor $\gamma$ of $\onion (O)$ has no $\omega$-chain.
\end{lemma}
\begin{proof}
Suppose for a contradiction that there is a sequence $\sv_{\sls 0}<\sv_{\sls 1}<\cdots$ of separations $\sv_{\sls n}\in\gamma^\ast$.
Note that $\vs<\vr$ with $\vs\in \onion_m(O)$ and $\vr\in \onion_n(O)$ implies $m<n$. 
Hence the function $g\colon\omega\to\omega$ assigning to each $n<\omega$ the least $k<\omega$ with $\vs_{\! n}\in\onion_k(O)$ is strictly decreasing in that $g(m)>g(n)$ for all $m<n$, contradicting that there are only finitely many natural numbers~$<g(0)$.
\end{proof}

The corridors of a \Onion\ usually stem from $\Sinf$-trees:

\begin{theorem}\label{thm: Peeling the onion}
Let $G$ be any graph, let $\vT$ be any regular tree set of finite-order separations of $G$, and let $O$ be any consistent orientation of $\vT$ such that $\vSinf [\onion(O)]$ is regular.
Then for every corridor $\gamma$ of the \Onion\ of $O$ the corresponding regular tree set $\vSinf[\gamma]$  is isomorphic to the edge tree set of a tree.
\end{theorem}
\begin{proof}
Let $\gamma$ be any corridor of the \Onion\  of $O$. 
By Theorem~\ref{KneipTreeSets}, it suffices to show that $\vSinf[\gamma]$ has no $(\omega+1)$-chain.
For this, suppose for a contradiction that $\vs_{\! 0}<\vs_{\! 1}<\cdots<\vs_{\! \omega}$ is an $(\omega+1)$-chain in $\vSinf[\gamma]$.

If $\vs_{\!\omega}$ lies in $\gamma$, then so do all the other $\vs_{\! n}$ as $\gamma$ is consistent.
Note that $\vs<\vr$ with $\vs\in \onion_m(O)$ and $\vr\in \onion_n(O)$ implies $m<n$.
Hence the function $f\colon\omega+1\to\omega$ assigning to each $\alpha\le\omega$ the least $n<\omega$ with $\vs_{\!\alpha}\in\onion_n(O)$ is strictly increasing in that $f(\alpha)<f(\beta)$ for all $\alpha<\beta$, contradicting $f(\omega)<\omega$.

Otherwise $\vs_{\! \omega}$ lies in $\gamma^\ast$.
If there is a number $N<\omega$ with $\vs_{\! n}\in\gamma^\ast$ for all $n\ge N$, without loss of generality $N=0$, then $\gamma^\ast$ has an $\omega$-chain contradicting Lemma~\ref{inversesCorridorsParliaments}.

Therefore, we may assume that $\vs_{\! n}\in\gamma$ for infinitely many $n<\omega$.
Since $\gamma$ is consistent, $\vs_{\! n}\in\gamma$ for all $n<\omega$ follows. 
Using that $\gamma$ is a corridor we find a separation $\vr\in\gamma$ with $\sv_{\sls\omega}\le\vr$ and $\vs_{\! 0}\le\vr$.
For every $n<\omega$, either $\vs_{\! n}\le\vr$ or $\vs_{\! n}\le\rv$ or $\sv_{\sls n}\le\vr$ or $\sv_{\sls n}\le\rv$. 
We cannot have $\vs_{\! n}\le\rv$ for any~$n$, since this would imply $\vs_{\! 0}<\vs_{\! n}\le\rv\le\sv_{\sls 0}$ contradicting that $\vSinf[\onion (O)]$ is regular.
We cannot have $\sv_{\sls n}\le\vr$ for any~$n$ because $\gamma$ is consistent.
And we cannot have $\sv_{\sls n}\le\rv$, because then $\sv_{\sls\omega}\le\vr\le\vs_{\! n}<\vs_{\!\omega}$ contradicts that $\vSinf[\onion (O)]$ is regular.
Hence $\vs_{\! n}\le\vr$ for all~$n$.
As $\gamma$ contains no $(\omega+1)$-chains by the first case, there must be an $\ell<\omega$ with $\vs_{\!\ell}=\vr$.
But this then contradicts $\vr=\vs_{\!\ell}<\vs_{\!\ell+1}\le\vr$, completing the proof that $\vSinf[\gamma]$ has no $(\omega+1)$-chains.
\end{proof}

Finally, we prove our technical theorem:

\begin{proof}[{Proof of Theorem~\ref{TechnicalStarDecomposition}}]
Let $(T_{\cK},\cY,\cK)$ be any \principal\ tree set of a connected graph $G$ so that $O_{\cK}$ defines an infinite part.
We let $O$ be the \Onion\ of $O_{\cK}$.
Then the tree set $\vSinf[O]$ is regular: for every $n\in\N$ and every $(A,B)\in\onion_n(O_{\cK})\subset O$ we have that $A\setminus B$ contains the non-empty vertex set of the graph $\bigcup\cK(X)$ for some $X\in\cY$, and $B\setminus A$ contains all but at most $\vert A\cap B\vert\le n$ of the infinitely many vertices of the infinite part defined by~$O$.
Therefore, by Theorem~\ref{thm: Peeling the onion} we find for every corridor $\gamma$ of $O$ an $\Sinf$-tree $(T_\gamma,\alpha_\gamma)$ such that $\alpha_\gamma$ is an isomorphism between the edge tree set $\vE(T_\gamma)$ of $T_\gamma$ and $\vSinf[\gamma]$.

In a first step, we will use the $\Sinf$-trees $(T_\gamma,\alpha_\gamma)$ to define stars $\sigma_\gamma$, one for every corridor $\gamma$ of $O$, such that their union $\sigma=\bigcup_\gamma\sigma_\gamma$ is a candidate for the star that we seek.
Then, in a second step, we will verify that $\sigma$ is indeed as desired, completing the proof.

First step.
We define stars $\sigma_\gamma$, one for each corridor $\gamma$ of $O$, such that their union $\sigma:=\bigcup_\gamma \sigma_\gamma$ is a candidate for the star that we seek.
For this, consider any corridor $\gamma$ of $O$.
Then $\gamma$, as it orients the image of $\alpha_\gamma$ consistently, defines either a node or an end of $T_\gamma$ (see Section~2.8 of~\cite{StarComb1StarsAndCombs}).

If $\gamma$ defines a node $t$ of $T_\gamma$, then $t$ has precisely one neighbour in $T_\gamma$.
Indeed, $\gamma$~is the down-closure in $\vSinf [\gamma]$ of the star $\alpha_\gamma(\vF_{\!t})$ where $\vF_{\! t}=\{\,(e,s,t)\in\vE(T_\gamma)\mid e=st\in T_\gamma\,\}$. 
Note that all separations in $\alpha_\gamma(\vF_{\! t})$ are maximal in $\gamma$.
Hence, if $t$ has two distinct neighbours $k_1$ and $k_2$ in $T_\gamma$, then $\gamma$ contains a separation $\vr$ that lies above both $\alpha_\gamma(k_1,t)$ and $\alpha_\gamma(k_2,t)$, contradicting the maximality in the corridor $\gamma$ of at least one of these two separations (here we also use that  $\alpha_\gamma(k_1,t)$ and $\alpha_\gamma(k_2,t)$ are distinct for distinct neighbours $k_1$ and $k_2$ of $t$ because $\alpha_\gamma$ is injective).
Therefore, $t$~is a leaf of $T_\gamma$.
Call its neighbour~$k$.
Then $\alpha_\gamma(k,t)$ is the maximal element of the corridor $\gamma$, and we let $\sigma_\gamma:=\{\,\alpha_\gamma (k,t)\,\}$.

Otherwise $\gamma$ defines an end of $T_\gamma$ from which we pick a ray $R_\gamma=v_\gamma^0 v_\gamma^1\ldots$ all whose edges are oriented forward by $\gamma$ in that $\vs_{\!\gamma}^{\srs n}:=\alpha_\gamma(v_\gamma^n,v_\gamma^{n+1})$ lies in $\gamma$ for all~$n\in\N$.
Then we let
\begin{align}\label{eq:tauGammaFromRay}
    \sigma_\gamma:=\{\,\vs_{\!\gamma}^{\srs 0}\,\}\cup\{\,\vs_{\!\gamma}^{\srs n}\wedge\sv_{\sls\gamma}^{n-1}:\, n\ge 1\,\}.
\end{align}
(See Figure~\ref{fig: folding separations}.)
\begin{figure}[h]
\centering
\def\svgwidth{\columnwidth}
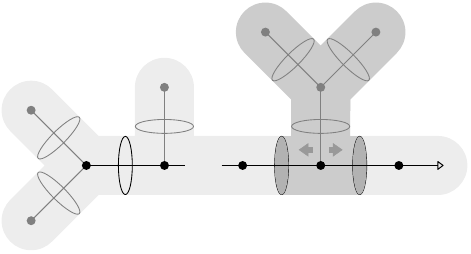
\caption{The light grey area depicts $B\setminus A$, the grey area depicts $A\setminus B$ and the dark grey area depicts $A\cap B$ of the separation $(A,B):=\vs_{\!\gamma}^{\srs n}\wedge\sv_{\sls\gamma}^{n-1}$ from the proof of Theorem~\ref{TechnicalStarDecomposition}.}
\label{fig: folding separations}
\end{figure}

 Let us check that $\sigma_\gamma$ really is a star. On the one hand, it follows from $\vs_{\!\gamma}^{\srs 0}\le\vs_{\!\gamma}^{\srs n-1}$ that $\vs_{\!\gamma}^{\srs 0}\le\sv_{\sls\gamma}^{n}\vee\vs_{\!\gamma}^{\srs n-1}=(\vs_{\!\gamma}^{\srs n}\wedge\sv_{\sls\gamma}^{n-1})^\ast$ for all $n\ge 1$.
And on the other hand, for $1\le n<m$, we infer from $\vs_{\!\gamma}^{\srs n-1}\le\vs_{\!\gamma}^{\srs n}\le\vs_{\!\gamma}^{\srs m-1}\le\vs_{\!\gamma}^{\srs m}$ that
\begin{align*}
    \vs_{\!\gamma}^{\srs m}\wedge\sv_{\sls\gamma}^{m-1}\le\sv_{\sls\gamma}^{m-1}\le\sv_{\sls\gamma}^{n}\le \sv_{\sls\gamma}^{n}\vee\vs_{\!\gamma}^{\srs n-1}= (\vs_{\!\gamma}^{\srs n}\wedge\sv_{\sls\gamma}^{n-1})^\ast.
\end{align*}
Since all $\vs_{\!\gamma}^{\srs n}$ have finite order, so do the infima of which~$\sigma_\gamma$ is composed.
This technique of turning a ray into a star of separations has been introduced by Carmesin~\cite{carmesin2014all} in  his `Proof that Lemma~6.8 implies Lemma~6.7'.

Second step. We prove that $\sigma$ is as desired. 
First, we show condition (i), 
which states that the part defined by $O_{\cK}$ is included in the part of $\sigma$. 
For every separation $\vs\in\sigma$ there is some separation $\vr\in O$ satisfying $\vs\le\vr$. Hence the part of $\sigma$ includes the part of $O$, which in turn includes the part of $O_{\cK}$ because $O$ is the \Onion\ of~$O_{\cK}$.

It remains to verify condition (ii), which states that for every $\tlsep{X}\in O_\cK$ there is some $\vs\in \sigma$ with $\tlsep{X}\tle\vs$.
For this, let any vertex set $X\in\cY$ be given.
As $O$ is cofinal in $O_{\cK}\cup O$, there is a separation $\vs_{\! X}\in O$ above $\tlsep{X}$.
Let $\gamma$ be the corridor of $O$ containing $\vs_{\! X}$.
We check the following two cases.

In the first case, $\sigma_\gamma$ is a singleton, formed by the maximal element $\vs$ of $\gamma$, giving
\begin{align*}
    \tlsep{X}\le\vs_{\! X}\le\vs\in\sigma.
\end{align*}

In the second case, $\sigma_\gamma$ is of the form~(\ref{eq:tauGammaFromRay}).
Then, as $O$ is nested with $T_{\cK}$, the separation $\tlsep{X}$ induces a consistent orientation of the image of $\alpha_\gamma$, as follows.
The orientation consists of all $\vr\in\vSinf[\gamma]$ that satisfy either $\vr\le\tlsep{X}$ or $\tlsep{X}<\rv$.
Now this consistent orientation defines either a node or an end of~$T_\gamma$.
Since $\vs_{\! X}\in\gamma$ lies above $\tlsep{X}$ and since $\gamma^\ast$ contains no $\omega$-chains by Lemma~\ref{inversesCorridorsParliaments}, it must be a node $t$ of $T_\gamma$.
Let $P=t_0\dots t_k$ be the $t$--$R_\gamma$ path in $T_\gamma$ and let $n\in\N$ be the number with $v_\gamma^n=t_k$, see Figure~\ref{fig: FigTechnicalThmP} (the ray $R_\gamma=v_\gamma^0 v_\gamma^1\ldots$ was defined right above~(\ref{eq:tauGammaFromRay})).

\begin{figure}[h]
\centering
\def\svgwidth{0.9\columnwidth}
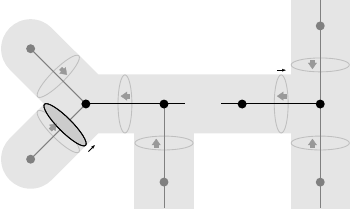
\caption{The orientation of the image $\vSinf [\gamma]$ of $\alpha_\gamma$ and the path $P$ in the second step of the proof of Theorem~\ref{TechnicalStarDecomposition}.}
\label{fig: FigTechnicalThmP}
\end{figure}

We claim that we may assume $n\neq 0$. For this, it suffices to show that we may assume that $\vs_{\!\gamma}^{\srs 0}$ lies in the orientation that defines $t$.
So let us consider the case that $\sv_{\sls\gamma}^0$ instead of $\vs_{\!\gamma}^{\srs 0}$ lies in the orientation that defines $t$.
In this case we have either $\sv_{\sls\gamma}^0\le\tlsep{X}$ or $\tlsep{X}<\vs_{\!\gamma}^{\srs 0}$. 
But actually, we cannot have $\sv_{\sls\gamma}^0\le\tlsep{X}$ because otherwise $\tlsep{X}\le\vs_{\! X}$ would imply that $\sv_{\sls\gamma}^0\le\vs_{\! X}$ meaning that $\sv_{\sls\gamma}^0$ and $\vs_{\! X}$ violate the consistency of $\gamma$.
Therefore, we must have $\tlsep{X}<\vs_{\!\gamma}^{\srs 0}$, and then we are done because $\vs_{\!\gamma}^{\srs 0}$ is an element of $\sigma_\gamma$.
Thus, we may assume $n>0$.

If the path $P$ is non-trivial, i.e., if $t_0=t$ is distinct from $t_k=v_\gamma^n$, then we consider the separation $\vr_{\! P}=\alpha_\gamma(t_{k-1},t_k)\in\gamma$ associated with the last edge $t_{k-1}t_k$ of~$P$. 
By the definition of $P$, the separation $\rv_{\! P}$ satisfies either $\rv_{\! P}\le\tlsep{X}$ or $\tlsep{X}<\vr_{\! P}$.
The former inequality would violate the consistency of $\gamma$ as $\rv_{\! P}\le\tlsep{X}\le\vs_{\! X}$ would follow (here we use that $\vSinf[\gamma]\subset\vSinf[O]$ is regular to ensure $\vr_{\! P}\neq\vs_{\! X}$).
Hence $\tlsep{X}<\vr_{\! P}$.
As $t_{k-1}$ is distinct from $v_\gamma^{n-1}$, and both vertices have  $v_\gamma^n$ as a neighbour in $T_\gamma$, we obtain the inequalities $\vr_{\! P}\le\vs_{\!\gamma}^{\srs n}$ and $\vr_{\! P}\le\sv_{\sls\gamma}^{n-1}$.
Thus, 
\begin{align*}
    \tlsep{X}\le\vr_{\! P}\le \vs_{\!\gamma}^{\srs n}\wedge\sv_{\sls\gamma}^{n-1}\in\sigma.
\end{align*}

Otherwise the path $P$ is trivial, i.e., $t_0=t_k$ where $t_0=t$ and $t_k=v_\gamma^n$.
By the definition of $t$ we have either $\vs_{\!\gamma}^{\srs n-1}\le\tlsep{X}$ or $\tlsep{X}<\sv_{\sls\gamma}^{n-1}$, and we have either $\sv_{\sls\gamma}^{n}\le\tlsep{X}$ or $\tlsep{X}<\vs_{\!\gamma}^{\srs n}$.
The case $\sv_{\sls\gamma}^{n}\le\tlsep{X}$ is impossible since otherwise $\tlsep{X}\le\vs_{\! X}\in\gamma$ would imply that $\sv_{\sls\gamma}^{n}\le\vs_{\! X}$ meaning that $\vs_{\!\gamma}^n$ and $\vs_{\! X}$ violate the consistency of $\gamma$.
Therefore, we have either $\tlsep{X}\le\vs_{\!\gamma}^{\srs n}\wedge\sv_{\sls\gamma}^{n-1}\in\sigma$ as desired, or we have $\vs_{\!\gamma}^{\srs n-1}\le \tlsep{X}<\vs_{\!\gamma}^{\srs n}$.
For this latter case, we show that there is a component $C\in\cK(X)$ such that $\vs_{\!\gamma}^{\srs n-1}\le\lsep{C}{X}$ holds.
This suffices to complete the proof, because then the inequalities $\lsep{\cK(X)\setminus\{C\}}{X}\le\lsep{X}{C}\le\sv_{\sls\gamma}^{n-1}$ and $\lsep{\cK(X)\setminus\{C\}}{X}\le\tlsep{X}<\vs_{\!\gamma}^{\srs n}$ give
\begin{align*}
    \lsep{\cK(X)\setminus\{C\}}{X}\le\vs_{\!\gamma}^{\srs n}\wedge\sv_{\sls\gamma}^{n-1}\in\sigma.
\end{align*} 

The separation $\vs_{\!\gamma}^{\srs n-1}\in O$ is, by definition, the supremum of some corridor $\delta$ of $\{\,(A,B)\in O_{\cK}:\vert A\cap B\vert\le\ell\,\}$ for some number $\ell\in \N$.
Then every separation $\tlsep{Y}\in \delta$ satisfies $\tlsep{Y}\le \vs_{\!\gamma}^{\srs n-1}\le \tlsep{X}$.
In particular, as the \principal\ tree set $T_{\cK}$ satisfies the conclusions of Theorem~\ref{TreeSetAPC}, every separation $\tlsep{Y}\in\delta$ satisfies $\tlsep{Y}\le\lsep{C_X(Y)}{X}$.
Hence in order to show that $\vs_{\!\gamma}^{\srs n-1}\le\lsep{C}{X}$ for some component $C\in\cK(X)$, it suffices to show that $C_X(Y)=C_X(Y')$ for every two separations $\tlsep{Y}$ and $\tlsep{Y'}$ in $\delta$.
Given $\tlsep{Y}$ and $\tlsep{Y'}$, consider any separation $\tlsep{Z}\in\delta$ above the two.
Then $\tlsep{Z}\le\lsep{C_X(Z)}{X}$ implies that both $C_X(Y)$ and $C_X(Y')$ are contained in $C_X(Z)$, giving $C_X(Y)=C_X(Y')$ as desired.
\end{proof}

\section{Overview of all duality results}\label{sec:summary}

\noindent In this section we summarise all duality theorems of this series.
A very brief overview of the complementary structures is given by the following table:\vspace{0.5cm}

\begin{center}
\begin{tabular}{r|c|c|c}
     & normal tree & tree-decomposition & other\\
\hline
combs & \cmark & \cmark & \cmark\\
stars & \cmark & \cmark\\
dominated combs & \cmark & \cmark\\
dominating stars & \cmark & \cmark\\
\hline
undominated comb & \xmark & \cmark & \cmark\\
undominating star & \xmark & \cmark & \cmark
\end{tabular}
\end{center}\vspace{0.5cm}
Here, a check mark means, for example, that we proved a duality theorem for combs in terms of normal trees, whereas the two crosses mean that normal trees cannot serve as complementary structures for undominated combs or undominating stars.

Finally, we summarise our duality theorem for combs, stars and combinations of the two explicitly in five theorems:

\newpage
\begin{NNtheorem}[Combs]
Let $G$ be any connected graph and let $U\subset V(G)$ be any vertex set.
Then the following assertions are equivalent:
\begin{enumerate}
\item $G$ does not contain a \dblue{comb} \at $U$;
\item there is a \dblue{rayless normal tree} $T\subset G$ that contains $U$ (moreover, $T$ can be chosen such that it contains $U$ cofinally);
\item $G$ has a \dblue{rayless tree-decomposition} into parts each containing at most finitely many vertices from $U$ and whose parts at non-leaves of the decomposition tree are all finite (moreover, the tree-decomposition displays $\Abs{U}$ and may be chosen with connected separators);
\item for every infinite $U'\subset U$ there is a \dblue{critical vertex set} $X\subset V(G)$ such that infinitely many of the components in $\CX$ meet $U'$;
\item $G$ has a \dblue{$U$-rank};
\item $G$ has a rooted tame \dblue{tree-decomposition} $(T,\cV)$ that covers $U$ cofinally and satisfies the following four assertions:
    \begin{itemize}
        \item[\textbf{--}] $(T,\cV)$ is the squeezed expansion of a normal tree of $G$ that contains the vertex set~$U$ cofinally;
        \item[\textbf{--}] every part of $(T,\cV)$ meets $U$ finitely and parts at non-leaves are finite;
        \item[\textbf{--}] $(T,\cV)$ displays $\AbsC{U}\subset\crit(G)$;
        \item[\textbf{--}] the rank of $T$ is equal to the $U$-rank of $G$.
    \end{itemize}
\end{enumerate}
\end{NNtheorem}


\begin{NNtheorem}[Stars]
Let $G$ be any connected graph and let $U\subset V(G)$ be any vertex set.
Then the following assertions are equivalent:
\begin{enumerate}
\item $G$ does not contain a \dblue{star} \at $U$;
\item there is a \dblue{locally finite normal tree} $T\subset G$ that contains $U$ and all whose rays are undominated in $G$ (moreover, $T$ can be chosen such that it contains $U$ cofinally and every component of $G-T$ has finite neighbourhood);
\item $G$ has a \dblue{locally finite tree-decomposition} with finite and pairwise disjoint separators such that each part contains at most finitely many vertices of~$U$ (moreover, the tree-decomposition can be chosen with connected separators and such that it displays $\AbsC{U}\subset\Omega(G)$);
\end{enumerate}
\end{NNtheorem}


\begin{NNtheorem}[Dominating stars and dominated comb]
Let $G$ be any connected graph and let $U\subset V(G)$ be any vertex set.
Then the following assertions are equivalent:
\begin{enumerate}
    \item $G$ does not contain a \dblue{dominating star} \at $U$;
    \item $G$ does not contain a \dblue{dominated comb} \at $U$;
    \item there is a \dblue{normal tree} $T\subset G$ that contains $U$ and all whose rays are undominated in $G$ (moreover, the normal tree $T$ can be chosen such that it contains $U$ cofinally and every component of $G-T$ has finite neighbourhood);
    \item $G$ has a rooted \dblue{tree-decomposition} $(T,\cV)$ such that
    \begin{itemize}
        \item[\textbf{--}]each part contains at most finitely many vertices from $U$;
        \item[\textbf{--}]all parts at non-leaves of $T$ are finite;
        \item[\textbf{--}]$(T,\cV)$ has essentially disjoint connected separators;
        \item[\textbf{--}] $(T,\cV)$ displays $\Abs{U}$.
    \end{itemize}
\end{enumerate}
\end{NNtheorem}

\newpage
\begin{NNtheorem}[Undominated combs]
Let $G$ be any connected graph and let $U\subset V(G)$ be any vertex set.
Then the following assertions are equivalent:
\begin{enumerate}
    \item $G$ does not contain an \dblue{undominated comb} \at $U$;
    \item $G$ has a \dblue{star-decomposition} with finite separators such that $U$ is contained in the central part and all undominated ends of $G$ live in the leaves' parts (moreover, the star-decomposition can be chosen with connected separators);
    \item $G$ has a connected \dblue{subgraph} that contains $U$ and all whose rays are dominated in it.
\end{enumerate}
Moreover, if $U$ is normally spanned in $G$, we may add
\begin{enumerate}[resume]
    \item there is a \dblue{rayless tree} $T\subset G$ that contains $U$.
\end{enumerate}
\end{NNtheorem}


\begin{NNtheorem}[Undominating stars]
Let $G$ be any connected graph and let $U\subset V(G)$ be any vertex set.
Then the following assertions are equivalent:
\begin{enumerate}
    \item $G$ does not contain an \dblue{undominating star} \at $U$;
    \item there is a tough \dblue{subgraph} $H\subset G$ that contains $U$;
    \item $G$ has a tame \dblue{star-decomposition} such that $U$ is contained in the central part and every critical vertex set of $G$ lives in a leaf's part.
\end{enumerate}
Moreover, if $U$ is normally spanned, we may add
\begin{enumerate}[resume]
    \item there is a locally finite \dblue{normal tree} $T\subset G$ that contains $U$.
\end{enumerate}
\end{NNtheorem}


\bibliographystyle{amsplain}
\bibliography{StarCombBib}
\end{document}